\documentclass[final,3p,times]{elsarticle}
\usepackage{amssymb,color,ulem}

\usepackage[usenames,dvipsnames]{xcolor}

\def\qed{\hfill $\square$}

\journal{}
\usepackage{amsthm,verbatim}
\theoremstyle{definition}
\newtheorem{definition}{Definition}[section]
\newtheorem{theorem}{Theorem}[section]
\newtheorem{cor}{Corollary}[section]
\newtheorem{lem}{Lemma}[section] 
\newtheorem{rem}{Remark}[section]

\newtheorem{ex}{Example}[section]
\newproof{pf}{Proof}
\newcommand{\eps}{\varepsilon}

\usepackage{anysize}  %----------------margins
\usepackage{bigstrut} % controling row heights of tables
\usepackage{mathrsfs}
 \usepackage{amsmath,amssymb}

\begin{document}
\begin{frontmatter}

\title{Global stability of almost periodic solutions of monotone sweeping processes and their response to non-monotone perturbations}
%\author[a]{Mikhail Kamenskii}

\author[mymainaddress]{Mikhail Kamenskii}

\author[mysecondaryaddress]{Oleg Makarenkov\corref{mycorrespondingauthor}}
\cortext[mycorrespondingauthor]{Corresponding author}
\ead{oxm130230@utdallas.edu}

\author[mysecondaryaddress]{Lakmi Niwanthi}

\author[mytertiaryaddress]{Paul Raynaud de Fitte}

\address[mymainaddress]{Department of Mathematics, Voronezh State University, Voronezh, Russia}
\address[mysecondaryaddress]{Department of Mathematical Sciences, University of Texas at Dallas, 75080 Richardson, USA}
\address[mytertiaryaddress]{Normandie University, Laboratoire Raphael Salem, UMR CNRS 6085, Rouen, France}

%\author{}%\corref{cor1}}
%\ead{}
\address{%BCAM - Basque Center for Applied Mathematics, Mazarredo 14, E-48009 Bilbao, Basque Country - Spain
}
\begin{abstract} 
We develop a theory which allows making qualitative conclusions about the dynamics of both monotone and non-monotone Moreau sweeping processes. Specifically, we first prove that any sweeping processes with almost periodic monotone right-hand-sides admits a globally exponentially stable almost periodic solution. And then we describe the extent to which such a globally stable solution persists under non-monotone perturbations.
\end{abstract}
\end{frontmatter}
% \linenumbers
%\tableofcontents
\section{Introduction}\label{sec:int} A perturbed Moreau sweeping process reads as
\begin{equation}\label{1}
   -\dot x(t)\in N_{C(t)}(x(t))+f(t,x(t)),
\end{equation}
where $N_C(x)$ is the outward normal cone 	
\begin{equation}\label{NC}
  N_C(x)=\left\{\begin{array}{ll}\left\{\xi\in\mathbb{R}^n:\left<\xi,c-x\right>\le 0,\ {\rm for\ any }\ c\in C\right\},& {\rm if}\ x\in C,\\
   \emptyset,& {\rm if}\ x\not\in C.
\end{array}\right.
\end{equation}
and $f:\mathbb{R}\times\mathbb{R}^n\times\mathbb{R}\to\mathbb{R}^n$  (see \cite{castaing,kunze,km,existence}). The unboundedness of the right-hand-sides in (\ref{1}) makes the classical theory of differential inclusions (see e.g. \cite{celina,koz}) inapplicable. 
 And despite numerous applications in elastoplasticity (see e.g. \cite{lam1,lam2}) (as well as in problems of power converters \cite{hybrid} and crowd motion \cite{crowd}), the theory of Moreau differential inclusions (also called {\it sweeping processes}) is  still in its infancy. Fundamental results on the existence, uniqueness and dependence of solutions on the initial data are proposed in Monteiro Marques \cite[Ch.~3]{mont}, Valadier \cite{valadier1}, Castaing and Monteiro Marques  \cite{castaing}, Adly-Le \cite{adly}, Brogliato-Thibault \cite{brogliato0}, Krejci-Roche \cite{roche}, Paoli \cite{paoli}. 
Dependence of solutions on parameters is covered in Bernicot-Venel \cite{venel0} and Kamenskiy-Makarenkov \cite{km}. The papers \cite{km,castaing} also show the existence of $T$-periodic solutions for $T$-periodic in time (\ref{1}).  Optimal control problems for sweeping process (\ref{1}) and equivalent differential equations with hysteresis operator are addressed in Edmond-Thibault \cite{edm}, Adam-Outrata \cite{adam}  (which also discusses applications to game theory),
Brokate-Krejci \cite{brokate}. Numerical schemes to compute the solutions of (\ref{1}) are discussed through most of the papers mentioned above.

\vskip0.2cm

\noindent Much less is known about the asymptotic behavior as $t\to\infty$. The known results in this direction are due to Leine and van de Wouw \cite{leine,leine2}, Brogliato \cite{brogliato1}, and Brogliato-Heemels  \cite{brogliato2}. Applied to a time-independent sweeping process (\ref{1}) the statements of
\cite[Theorem~8.7]{leine} (or \cite[Theorem~2]{leine2}),
 \cite[Lemma~2]{brogliato1}, and \cite[Theorem~4.4]{brogliato2} imply the incremental stability and global exponential stability of an equilibrium, provided that
%\cite[Theorem~8.7]{leine} (or \cite[Theorem~2]{leine2}), Brogliato \cite[Lemma~2]{brogliato1}, and Brogliato-Heemels \cite[Theorem~4.4]{brogliato2}. 
%All these papers build upon monotonicity of the map $x\mapsto N_{C(t)}(x)$, i.e. 
%$$
 %  \left<y_1-y_2,x_1-x_2\right>\ge 0,\quad\mbox{for all}\ y_1\in N_{C(t)}(x_1),\ y_2\in N_{C(t)}(x_2),\ x_1,x_2\in\mathbb{R}^n
%$$
  \begin{equation}\label{monot}
    \langle f(t,x_1)-f(t,x_2),x_1-x_2 \rangle  \geq \alpha \lVert x_1-x_2 \lVert ^2,\quad\mbox{for some fixed}\ \alpha>0\ \mbox{and for all}\ t\in\mathbb{R},\ x_1,x_2\in\mathbb{R}^n.
  \end{equation}
In particular, the results of \cite{leine,leine2,brogliato1,brogliato2} do not impose any Lipschitz regularity on  $x\mapsto f(t,x)$ and the derivative in (\ref{1}) is a differential measure, which is capable to deal with solutions $x$ of bounded variation.

\vskip0.2cm

\noindent This paper is motivated by sweeping processes (\ref{1}) coming from models of parallel networks of elastoplastic springs (see e.g. Bastein et al \cite{lam1,lam2}), where the right-hand-sides are Lipschitz in all the variables. 
 Here $C(t)$ represents the mechanical loading of the springs and $f(t,x)$ stands for those forces which influence the masses of nodes. 
Time-periodically changing $C$ and $f$ are most typical in laboratory experiments 
 (see \cite{20,exp2,lam1}). However, the different nature of
$t\mapsto C(t)$ and $t\mapsto f(t,x)$
 makes it most reasonable to not rely on the existence of a common period when the two functions 
 receive periodic excitations, but rather to use a theory which is capable to deal with arbitrary different periods of $t\mapsto C(t)$ and $t\mapsto f(t,x)$. The goal of this paper is to develop such a theory.

\vskip0.2cm

\noindent  Specifically, by assuming that both 
$t\mapsto C(t)$ and $t\mapsto f(t,x)$ are almost periodic, we establish global exponential stability of an almost periodic solution to a monotone sweeping process (\ref{1eps}). The corresponding theory for differential equations is available e.g. in Trubnikov-Perov \cite{trub} and Zhao \cite{jde}, that found numerous applications in biology. Moreover, we show that the almost periodic solution found preserves its stability under a wide class of non-monotone  perturbations, which is known for differential inclusions with bounded right-hand-sides e.g. from Kloeden-Kozyakin \cite{kloeden1} and Plotnikov \cite{plot}.

\vskip0.2cm

\noindent The paper is organized as follows. Section 2 establishes (Theorem~\ref{thm1}) the existence of a  solutions to (\ref{1}) defined on the entire $\mathbb{R}$ under the assumption that both $t\mapsto C(t)$ and $t\mapsto f(t,x)$ are uniformly bounded Lipschitz functions, but without any use of the monotonicity assumption (\ref{monot}). Note, that for any solution to (\ref{1}), $x(t)\in C(t)$, so any solution to (\ref{1}) is uniformly bounded in the domain of its definition. When the monotonicity 
assumption (\ref{monot}) holds, we have (Theorem~\ref{thm11}) the uniqueness and global exponential stability of a solution defined on the entire $\mathbb{R}.$ 
This result doesn't follow from 
\cite{brogliato1,brogliato2}, where the existence of an equilibrium follows from the particular structure of the right-hand-sides. When both $C(t)$ and $f(t,x)$ are constant in $t$, the existence of an equilibrium to (\ref{1}) formally follows from \cite{leine,leine2} which could transform into a  solution on $\mathbb{R}$ when  $C(t)$ and $f(t,x)$ get time-varying and uniformly bounded. We provide an independent proof because the proofs of  
\cite[Theorem~8.7]{leine} and \cite[Lemma~2]{leine2} rely on Yakubovich \cite[Lemma~2]{yak}. In turn, \cite[Lemma~2]{yak} sends the reader to Budak \cite[Theorem~2]{budak} for the most crucial step of the proof, which is compactness of a sequence $\{x_k\}_{k=1}^\infty$ of $C^0(\mathbb{R},\mathbb{R}^n)$ solutions of (\ref{1}) corresponding to a converging sequence of initial conditions. Even if one ignores verifying the regularity assumption of Budak \cite[Theorem~2]{budak}, this theorem provides a convergent subsequence on finite interval and Yakubovich \cite[Lemma~2]{yak} doesn't explain how the convergence gets extended to the entire $\mathbb{R}.$

\vskip0.2cm

\noindent Under the assumption that both $t\mapsto C(t)$ and $t\mapsto f(t,x)$ are almost periodic functions and $x\mapsto f(t,x)$ is monotone in the sense of (\ref{monot}), Section~3 shows (Theorem~\ref{thm2}) that the unique solution defined in Section~2 on the entire $\mathbb{R}$ is almost periodic. Here we follow Vesely \cite{vesely} to introduce the concept of almost periodicity for set-valued functions and for the respective Bochner's theorem. The results of \cite{vesely} are developed for functions with values in an arbitrary complete metric space and we take advantage of the completeness of the space of convex closed non-empty sets equipped with the Hausdorff metric (see e.g. Price \cite{price}) to apply Vesely's theory to sweeping processes.
The overall strategy of section~3 originates from the corresponding theory available for differential equations (see e.g. Trubnikov-Perov \cite{trub}).

\vskip0.2cm

\noindent Section~4 considers the sweeping process (\ref{1}) with a parameter $\eps$ under the assumption that the monotonicity condition (\ref{monot}) and almost periodicity of $C$ and $f$ only hold for $\eps=\eps_0$. When $\eps=\eps_0$, the sweeping process has an unique almost periodic solution $x_0$ by Theorem~\ref{thm2}. 
The result of section~4 (Theorems~\ref{thm4} and \ref{icthm}) proves that the solutions to the perturbed sweeping process with $\eps\not=\eps_0$ and with an initial condition $x_\eps(0)\in C(0)$ approach any given inflation of the solution $x_0$ (as it is termed in Kloeden-Kozyakin \cite{kloeden1}) when the values of time become large and when $\eps$ approaches $\eps_0$. Instructive examples of Section~\ref{examples} illustrate the domains of applications of Theorems~\ref{thm4} and \ref{icthm}. Finally, Section~\ref{newsec} gives a brief outlook about the potential role of Theorems~\ref{thm4} and \ref{icthm} in the analysis of the dynamics of networks of elastoplastic springs that motivated our study.
	
\vskip0.2cm

\noindent We note that condition (\ref{monot}) ensures that the sweeping process (\ref{1}) is incrementally stable (see  \cite[Theorem~8.7]{leine}, \cite[Lemma~2]{leine2}, or Theorem~\ref{thm11} below), which concept currently attracts an increasing attention in the switched systems literature, see e.g. Lu-di Bernardo \cite{switch1}, Zamani-van de Wouw-Majumdar \cite{switch2} and references therein. The source for incremental stability in the later papers lies, however, in contraction properties of the right-hand-sides (due to Demidovich, see \cite[Ch.~IV, \S16]{dem} and \cite{pavlov}), while the monotonicity property (\ref{monot}) ensures expansion.

\section{The existence of an unique globally exponentially stable bounded solution $x_0$}

\noindent Let $f:\mathbb{R}\times\mathbb{R}^n\to\mathbb{R}^n$ be globally Lipschitz continuous in the sense that 
\begin{equation}\label{Lipf}
  \|f(t_1,x_1)-f(t_2,x_2)\|\le L_f\|t_1-t_2\|+L_f\|x_1-x_2\|,\qquad {\rm for\ all\ }t_1,t_2\in\mathbb{R},\ x_1,x_2\in\mathbb{R}^n, \mbox{ and for some }L_f>0,
\end{equation}
A similar property 
\begin{equation}\label{Lip}
  d_H(C(t_1),C(t_2))\le L_C|t_1-t_2|,\qquad \mbox{for all}\ t_1,t_2\in\mathbb{R}, \mbox{ and for some }L_C>0,
\end{equation}
is assumed for the closed convex-valued function $t\mapsto C(t)$, where  the Hausdorff distance $d_H(C_1,C_2)$ between two closed sets $C_1,C_2 \subset \mathbb{R}^n$ is defined as
\begin{equation}\label{dH}
   d_H(C_1,C_2)=\max\left\{  
\sup_{x\in C_2} {\rm dist}(x,C_1),\sup_{x\in C_1} {\rm dist}(x,C_2)
\right\}\quad{\rm with}\quad {\rm dist}(x,C)=\inf\left\{
|x-c|:c\in C
\right\}.
\end{equation}
Under conditions (\ref{Lipf}) and (\ref{Lip}), for any initial condition $x(t_0)\in C(t_0)$,  the sweeping process (\ref{1}) with nonempty, closed and convex $C(t),$ $t\in\mathbb{R},$ admits (Edmond-Thibault \cite[Theorem~1]{existence}) a unique absolutely continuous forward solution $x(t)$, in the sense that $x(t)$ satisfies (\ref{1}) for almost all $t\ge t_0.$ 
\begin{rem} If $x_0$ is a solution to (\ref{1}) defined on $t\ge t_0$,  then $x(t)\in C(t),$ for all $t\ge t_0,$ because $N_{C(t)}(x(t))$ is undefined otherwise  (the interested reader can see that \cite{existence} obtains the solution $x(t)$  as $x(t)=y(t)-\psi(t)$, where $y(t)\in C(t)+\psi(t)$ \cite[pp.~352--353]{existence}). In particular, if 
$\|C(t)\|\le M$ for some $M>0$ and all $t\in\mathbb{R}$, then 
\begin{equation}\label{M}
    \|x(t)\|\le M,\quad{\rm  for\ any\ solution\ }x\ {\rm of\ }(\ref{1}){\rm\ with\ the\ initial\ condition\ }x(t_0)\in C(t_0){\rm \ and\ }t\ge t_0.   
\end{equation}  
\end{rem}

\begin{theorem}\label{thm1} Let $f:\mathbb{R}\times\mathbb{R}^n\times\mathbb{R}\to\mathbb{R}^n$ satisfy the Lipschitz condition (\ref{Lipf}). Assume that, for any $t\in\mathbb{R}$, the set $C(t)\subset\mathbb{R}^n$ is nonempty, closed, convex and the map $t\mapsto C(t)$ satisfies the Lipschitz condition (\ref{Lip}). 
If  $C$ is globally bounded, then the sweeping process (\ref{1}) admits at least one absolutely continuous solution $x_0$ defined on the entire $\mathbb{R}$. 
\end{theorem}

  \begin{proof} 
 \noindent {\bf Step 1:} {\it Construction of a candidate solution $x_0$ bounded on the entire $\mathbb{R}$}.
Let $\{\xi_m\}_{m=1}^\infty$ be an arbitrary sequence of 
elements of $\mathbb{R}^n$ such that 
$\xi_m\in C(-m)$, $m\in\mathbb{N}.$  Let $x_m(t)$ be %$u(t;-n,\xi_n)$,be 
the solution to (\ref{1}) with the initial condition $x_m(-m)=\xi_m.$ Extend each $x_m$ from $[-m,\infty)$ to $\mathbb{R}$ by defining $x(t)=x(-m)$ for all $t<-m.$ By Thibault, the functions of $\{x_m(t)\}_{m=1}^\infty$ share same Lipschitz constant $L_0>0$. Therefore, for each $k\in\mathbb{N}$ we can extract a subsequence $\{x_m^k(t)\}_{m=1}^\infty$ which converges uniformly on $[-k,k]$. By using these family of subsequence we introduce a sequence $\{x_m^*\}_{m=1}^\infty$ by $x_m^*(t)=x_m^m(t).$ The sequence $\{x_m^*\}_{m=1}^\infty$ converges uniformly on any fixed interval $[-k,k],$ $k\in\mathbb{N}.$
Define $x_0(t)$ by $x_0(t)=\lim\limits_{m\to\infty}x_m^*(t).$ The function $x_0:\mathbb{R}\to\mathbb{R}^n$ is Lipschitz continuous with constant $L_0.$

\vskip0.2cm

\noindent {\bf Step 2:} {\it Proof that $x_0$ is indeed a solution}. 
  	Let $\tau\in\mathbb{R}$ and let  $v$ be a solution of (\ref{1}) with $v(\tau)=x_0(\tau)$. Assume $ v(t_0) \neq x_0(t_0)$ for some $t_0 > \tau$, i.e. $ \lim\limits_{m\to\infty} x_m(t_0) \neq v(t_0)$. Then there exists $ \varepsilon_0 > 0 $, such that for each $m\in \mathbb{N}$, there exists $m_n>m$ such that $ \lVert x_{m_n}(t_0) - v(t_0)\lVert \geq \varepsilon_0 $. On the other hand, by continuous dependence of solutions of (\ref{1}) on the initial condition (see Edmond-Thibault \cite[Proposition~2]{existence}), there exists $ \delta > 0$ such that if $ \lVert v(\tau) - x_m(\tau) \lVert < \delta$ then $ \lVert v(t) - x_m(t) \lVert < \varepsilon_0 $ for all $m \in \mathbb{N}$ with $-m<\tau$ (which ensures that $x_m(t)$ is a solution of (\ref{1}) for $t\ge \tau$) and $t\in[\tau,t_0]$. But since $v(\tau)=x_0(\tau) = \lim\limits_{n\to\infty} x_m(\tau) $, there exists $N \in \mathbb{N}$ such that $ \lVert v(\tau) - x_m(\tau) \lVert < \delta $ for each $m>N$. Then $ \lVert v(t) - x_m(t) \lVert < \varepsilon_0 $ for all $m >N$ and $t\in[\tau,t_0]$. This contradicts $ \lim\limits_{n\to\infty} x_m(t_0) \neq v(t_0)$. Therefore $ v(t) = x_0(t)$ for each $t\geq \tau$. Hence $x_0$ is a solution of (\ref{1}).  

\begin{figure}[h]\center
\includegraphics[scale=0.7]{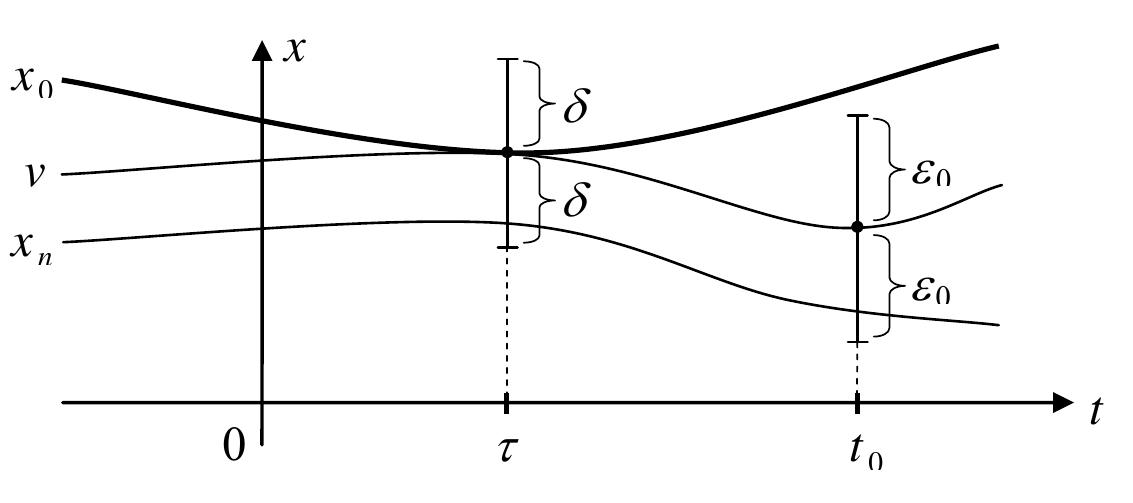}
\vskip-0.4cm
\caption{\footnotesize Illustration of the location of curves $x_0,$ $v,$ and $x_m$.} \label{fig2}
\end{figure}

\end{proof}

\vskip0.2cm

\begin{theorem}\label{thm11} Assume that the conditions of Theorem~\ref{thm1} hold. If  $f$ satisfies the monotonicity condition (\ref{monot}) then (\ref{1}) is incrementally stable and (\ref{1}) admits exactly one absolutely continuous solution $x_0$ defined on the entire $\mathbb{R}.$  Moreover, $x_0$ is globally exponentially stable. 
\end{theorem}

\noindent The incremental stability of (\ref{1}) under condition (\ref{monot}) is proved in \cite[Theorem~8.7]{leine} and \cite[Lemma~2]{leine2}. The statements about uniqueness and global stability of the bounded solution $x_0$ follow from incremental stability. We include a proof of Theorem~\ref{thm11} in Appendix for completeness.

\section{Almost periodicity of the bounded solution $x_0$}

  \begin{definition}
 Let $ck(\mathbb{R}^n)$ be the space of all closed bounded non-empty sets of $\mathbb{R}^n$ equipped with the Hausdorff metric $d_H$, see (\ref{dH}). A continuous function $\phi: \mathbb{R} \to (\mathbb{R}^n,d_H)$ is {\it almost periodic}, if for any $\varepsilon >0$, there exists a number $p(\varepsilon)>0 $ with the property that any interval of length $p(\varepsilon)>0 $ of the real line contains at least one point $s$, such that 
 \begin{equation*}
 d_H(\phi(t+s),\phi(t)) < \varepsilon \text{  for } t\in \mathbb{R}.
 \end{equation*} 
 
   \end{definition}
  
 \begin{theorem}\label{thm2} Let the conditions of  Theorem~\ref{thm1} hold and let $x_0$ be the unique absolutely continuous solution given by Theorem~\ref{thm1}.
If both the function $t\mapsto f(t,x)$ and the set-valued function $t\mapsto C(t)$ are 
almost periodic, then $x_0$ is almost periodic.
\end{theorem}
  \begin{proof}  Let $\{h_m\}_{m=1}^\infty \subseteq \mathbb{R}$. We are going to prove that there exists $\{k_m(x)\}_{m=1}^\infty \subseteq \{h_m\}_{m=1}^\infty$ such that the sequence of
\begin{equation}\label{xm}
x_m(t)=x_0(t+k_m),\quad m\in N,\ t\in\mathbb{R},
\end{equation} converges as $m\to\infty$ uniformly in $t\in\mathbb{R}$, which will  imply almost periodicity of $x_0$ by Bochner's theorem (see e.g. Levitan-Zhikov \cite[p.~4]{levitan}). 

\vskip0.2cm

\noindent {\bf Step 1.} {\it The existence of 
$\{l_m\}_{m=1}^\infty \subseteq \{h_m\}_{m=1}^\infty$ such that 
$f_m(t,x)=f(t+l_m,x)$  converges as $m\to\infty$ uniformly}.
  	Since $f(t,x)$ is almost periodic, then, for each $x\in\mathbb{R}^n$,  Bochner's theorem (see e.g. Levitan-Zhikov \cite[p.~4]{levitan}) implies the existence of 
$\{l_m(x)\}_{m=1}^\infty \subseteq \{h_m\}_{m=1}^\infty % \subseteq \{h_m\}_{m=1}^\infty 
$ such that the sequence of functions $\{f(\cdot+l_m(x),x)\}_{m=1}^\infty$ converges in the sup-norm. The standard diagonal method allows to construct $\{l_m(x)\}_{m=1}^\infty$ independent on $x.$ Indeed, considering $\{x_m\}_{m=1}^\infty=\mathbb{Q}^n$, we first construct  sequences $\{l_m(x_1)\}_{m=1}^\infty\supseteq \{l_m(x_2)\}_{m=1}^\infty\supseteq \ldots$, such that each individual sequence 
$\{f(\cdot+l_m(x_1),x_1)\}_{m=1}^\infty$,
$\{f(\cdot+l_m(x_2),x_2)\}_{m=1}^\infty$, $\ldots$ converges.
And then define $\{l_m\}_{m=1}^\infty\subseteq \{h_m\}_{m=1}^\infty$ as $l_m=l_m(x_m),$ $m\in\mathbb{N}.$ Put 
\begin{equation}\label{fm}
f_m(t,x)=f(t+l_m,x), \quad\mbox{for all}\ t\in\mathbb{R},\ x\in\mathbb{Q}^n,\ m\in\mathbb{N}.
\end{equation}
So constructed, $\{f_m(\cdot,x)\}_{m=1}^\infty$ converges for each fixed $x\in\mathbb{Q}^n$. Let
\begin{equation}\label{let}
  \hat f(t,x)=\lim\limits_{m\to\infty}f_m(t,x),\quad\mbox{for all}\ t\in\mathbb{R},\ x\in\mathbb{Q}^n.
\end{equation}
By (\ref{Lipf}) both $f_m$ and $\hat f$ are Lipschitz continuous with constant $L_f$ on 
$\mathbb{R}\times\mathbb{R}^n$ and $\mathbb{R}\times\mathbb{Q}^n$ respectively. Now we extend $\hat f$ from $\mathbb{R}\times\mathbb{Q}^n$ to $\mathbb{R}\times\mathbb{R}^n$ by taking an arbitrary sequence $\mathbb{Q}\ni x_k\to x_0\in\mathbb{R}$, as $k\to\infty,$ and defining $\hat f(t,x_0)=\lim\limits_{k\to\infty}\hat f(t,x_k).$ The limit exists because $\{\hat f(t,x_k)\}_{k=1}^\infty$ is a Cauchy sequence for each fixed $t\in\mathbb{R}$, which follows from Lipschitz continuity of $\hat f$ on $\mathbb{R}\times\mathbb{Q}^n$. Lipschitz continuity of $\hat f$ extends from $\mathbb{R}\times\mathbb{Q}^n$ to $\mathbb{R}\times\mathbb{R}^n$ by continuity. 
The latter property also implies that 
$$
  \left\|\hat f(t,x_0)-\hat f(t,x_k)\right\|\le L_f\|x_0-x_k\|,\quad\mbox{for all}\ k\in\mathbb{N}.
$$
Finally, to show that
\begin{equation}\label{fconverge}
   f_m(t,x)\to \hat f(t,x)\quad{\rm as}\ m\to\infty,\ \mbox{uniformly in }t\in\mathbb{R},\ x\in\mathbb{R}^n,
\end{equation}
we estimate $f_m(t,x)-\hat f(t,x)$ as
$$
  \left\|f_m(t,x)-\hat f(t,x)\right\|\le\left\|f_m(t,x)-f_m(t,x_*)\right\|+\left\|f_m(t,x_*)-\hat f(t,x_*)\right\|+\left\|\hat f(t,x_*)-\hat f(t,x)\right\|.
$$
Given $x\in\mathbb{R}$ and $\eps>0$, we choose $x_*\in\mathbb{Q}$ so close to $x$ that $\left\|f_m(t,x)-f_m(t,x_*)\right\|<\eps/3$ and $\left\|\hat f(t,x_*)-\hat f(t,x)\right\|<\eps/3$, for all $m\in\mathbb{N},$ $t\in\mathbb{R}.$ By (\ref{let}) we can now select $m_0\in\mathbb{N}$ such that $\left\|f_m(t,x_*)-\hat f(t,x_*)\right\|< \eps/3$, for all $m>m_0$ and $t\in\mathbb{R}.$ Thus, (\ref{fconverge}) holds.

\vskip0.2cm

\noindent {\bf Step 2.}  {\it The existence of 
$\{k_m\}_{m=1}^\infty \subseteq \{l_m\}_{m=1}^\infty$, such that 
$C_m(t)=C(t+k_m)$ converges as $m\to\infty$ uniformly}. By Bochner's theorem for almost periodic functions in pseudo-metric spaces (see \cite[Theorem~2.4]{vesely}), there exists $\{k_m\}_{m=1}^\infty \subseteq \{l_m\}_{m=1}^\infty$, such that $\{C_m(t)\}_{m=1}^\infty$ is a Cauchy sequence 
in $ck(\mathbb{R}^n)$, which is uniform in $t\in\mathbb{R}.$ The convergence 
of $\{ C_m(t)\}_{m=1}^\infty$ for each individual $t\in\mathbb{R}$ now follows from the completeness of $ck(\mathbb{R}^n)$ 
(Price \cite[the theorem of \S3]{price}). The uniformity of the convergence in $t\in\mathbb{R}$ follows along the standard lines. Indeed, let
$$
  \hat{C}(t)=\lim\limits_{m\to\infty} C_m(t). 
$$
Given $\eps>0$, 
fix $m_0>0$ such that $d_H(C_{m}(t),C_{m_*}(t))<\eps/2$ for all $m>m_0,$ $m_*>m_0$, and $t\in\mathbb{R}.$ For each $t\in\mathbb{R}$ select  
$m_*(t)>m_0$ such that $d_H\left(C_{m_*(t)}(t),\hat C(t)\right)<\eps/2.$ Then
$$
   d_H\left(C_m(t),\hat C(t)\right)\le d_H\left(C_m(t),C_{m_*(t)}(t)\right)+d_H\left(C_{m_*(t)}(t),\hat C(t)\right)<\eps/2+\eps/2=\eps,\quad \mbox{for all } \ m>m_0,\ t\in\mathbb{R}.
$$
Note that (\ref{Lip}) implies that  $\hat{C}$ is globally Lipschitz continuous with constant $L_C$. 

\vskip0.2cm

\noindent {\bf Step 3:} {\it The uniform convergence of $\{x_m(t)\}_{m=1}^\infty$}. 
  	The function $x_m$, see (\ref{xm}),  is a solution to the sweeping process
  \begin{equation}\label{sp1}
  -\dot{x}(t) \in N_{C_m(t)}(x(t))+f_m(t,x(t)).
  \end{equation} 
Along with (\ref{sp1}) let us consider
  	\begin{equation}\label{sp2}
  	-\dot{x}(t) \in N_{\hat{C}(t)}(x(t))+\hat{f}(t,x(t)).
  	\end{equation}  
Both $\hat C$ and $\hat f$ are globally bounded  and globally Lipschitz continuous. Moreover, by using  (\ref{fm}) and (\ref{let}) one concludes that $\hat f$ satisfies the monotonicity property (\ref{monot}). Therefore, by Theorem~\ref{thm1} the sweeping process (\ref{sp2}) has a unique bounded absolutely continuous solution $\hat x$ defined on the entire $\mathbb{R}.$
 	Let $t\in \mathbb{R}$ be such that both $ \dot{x}_m(t) $ and $ \dot{\hat{x}}(t) $ exist and satisfy the respective relations (\ref{sp1}) and (\ref{sp2}). Define 
$$v_m=\dot{x}_m(t) + f_m(t,x_m(t)),\ \hat v=\dot{\hat x}(t) + \hat f(t,\hat x(t)), \mbox{ so that }v_m\in -N_{C_m(t)}(x_m(t)),\ \hat{v}\in -N_{\hat{C}(t)}(\hat{x}(t)).$$
Furthermore, introducing $\Delta_m(t) = d_H\left(C_m(t),\hat{C}(t)\right)$ one has
$$
  	u_m(t)\in C_m(t) \subseteq \hat{C}(t) + \bar{B}_{\Delta_m(t)}(0),\ \ \hat{u}(t) \in \hat{C}(t) \subseteq C_m(t) + \bar{B}_{\Delta_m(t)}(0),\quad{\rm for\ all }\ t\in \mathbb{R}.$$
Therefore, $x_m$ and $\hat x$ can be decomposed as
$$
  x_m(t)=\hat d(t)+s_m(t),\ \ \hat x(t)=d_m(t)+\hat s(t),\quad{\rm where}\ 
\hat{d}(t) \in \hat{C}(t),\ d_m(t) \in C_m(t),\ 
\lVert s_m(t) \lVert\leq \Delta_m(t),\ \lVert \hat{s}(t) \lVert \leq \Delta_m(t). 
$$
Let $$w_m(t)=\lVert x_m(t)-\hat{x}(t) \lVert ^2.$$ Then,
  	\begin{eqnarray*}
  	\frac{1}{2} \dot{w}_m(t) 
  	&=& \langle \dot{x}_m(t)-\dot{\hat{x}}(t) , x_m(t)-\hat{x}(t)\rangle  \\
  	&=&\langle  v_m(t)-f_m(t,x_m(t))-\hat{v}(t)+\hat{f}(t,\hat{x}(t)) , x_m(t)-\hat{x}(t) \rangle \\	
%   &&=\langle  v_n-\hat{v} , x_n(t)-\hat{x}(t) \rangle - \langle f_n(t,x_n(t))-\hat{f}(t,\hat{x}(t)) , x_n(t)-\hat{x}(t) \rangle \\
   &=&\langle  v_m(t) , x_m(t)-d_m(t) - \hat{s}(t) \rangle + \langle \hat{v}(t) , \hat{x}(t)-\hat{d}(t) - s_m(t) \rangle  - \langle f_m(t,x_m(t))-\hat{f}(t,\hat{x}(t)) , x_m(t)-\hat{x}(t)\rangle 
	\end{eqnarray*}
By (\ref{NC}) we have $\langle  v_m(t) , x_m(t)-d_m(t) \rangle\le 0$ and $\langle  \hat v(t) , \hat x(t)-\hat d(t) \rangle\le 0$. Therefore, for a.a. $t\in\mathbb{R},$ 
  	\begin{eqnarray*}
  	\frac{1}{2} \dot{w}_m(t) 
%   &&=\langle  v_n , x_n(t)-d_n \rangle  - \langle v_n , \hat{s} \rangle +  \langle \hat{v} , \hat{x}(t)-\hat{d} - s_n \rangle - \langle f_n(t,x_n(t))-\hat{f}(t,\hat{x}(t)) , x_n(t)-\hat{x}(t)\rangle \\
%   &&=\langle  v_n , x_n(t)-d_n \rangle  +  \langle \hat{v} , \hat{x}(t)-\hat{d} \rangle - \langle  v_n , \hat{s} \rangle -\langle \hat{v} , s_n\rangle - \langle f_n(t,x_n(t))-\hat{f}(t,\hat{x}(t)) , x_n(t)-\hat{x}(t)\rangle\\
   &\leq& - \langle  v_m(t) , \hat{s}(t) \rangle  -\langle \hat{v}(t) , s_m(t)\rangle - \langle f_m(t,x_m(t))-\hat{f}(t,\hat{x}(t)) , x_m(t)-\hat{x}(t)\rangle \\
  	&\leq& \lVert v_m(t)\lVert\cdot \lVert \hat{s}(t)\lVert + \lVert \hat{v}(t)\lVert\cdot  \lVert s_m(t)\lVert - \langle f_m(t,x_m(t))-f_m(t,\hat x(t))+f_m(t,\hat x(t))-\hat{f}(t,\hat{x}(t)) , x_m(t)-\hat{x}(t)\rangle.
\end{eqnarray*}
Given $\eps>0$ we use the conclusions of Steps 1 and 2 to spot an $m_0>0$ such that 
$$
   \|\hat s(t)\|\le \eps_0,\    \|s_m(t)\|\le \eps_0,\ \left\|f_m(t,\hat x(t))-\hat f(t,\hat x(t))\right\|\le\eps_0,\quad\mbox{for all}\ m\ge m_0,\ t\in\mathbb{R}^n.
$$
By Edmond-Thibault \cite[Theorem~1]{existence}, there exists $L_0>0$ such that
$$
   \|v_m(t)\|\le L_0,\ \|\hat v(t)\|\le L_0
$$
and by using (\ref{M}) we can estimate $\dot w_m(t)$ further as
\begin{eqnarray*}
\dfrac{1}{2}\dot w_m(t)&\le&  2\eps L_0-\left<f_m(t,x_m(t))-f_m(t,\hat x(t)),x_m(t)-\hat x(t)\right>+2\eps M,\quad\mbox{for all }m\ge m_0,\ a.a.\ t\in\mathbb{R}.
\end{eqnarray*}
By referring to the definition (\ref{fm}) of $f_m$, one observes that $f_m$ satisfies the monotonicity estimate (\ref{monot}), which implies
\begin{eqnarray*}
\dfrac{1}{2}\dot w_m(t)&\le&  2\eps (L_0+M)-\alpha\|x_m(t)-\hat x(t)\|^2=2\eps (L_0+M)-\alpha w_m(t),\quad\mbox{for all }m\ge m_0\ {\rm and \ a.a.\ }t\in\mathbb{R}.
\end{eqnarray*}
Gronwall-Bellman lemma (see Lemma~\ref{trub} in the Appendix) now allows to conclude that $$w_m(t)\leq w_m(\tau) e^{-\alpha (t-\tau)} + 2\eps(L_0+M)\int_{\tau}^{t} e^{-\alpha (t-s)}  ds=w_m(\tau) e^{-\alpha (t-\tau)} + \eps\dfrac{2(L_0+M)}{\alpha}\left(1-e^{-\alpha(t-\tau)}\right),\ \ t,\tau\in\mathbb{R},\ m\ge m_0. $$
By passing to the limit as $\tau\to-\infty$ one gets
$$	
  w_m(t)\le \eps \cdot 2(L_0+M)/\alpha,\quad{t\in\mathbb{R}},\ m\ge m_0.
$$
Therefore, $\lVert x_m(t)-\hat{x}(t) \lVert\to 0$ as $m\to\infty$ uniformly in $t\in\mathbb{R},$ and so $x_0$ is almost periodic by Bochner's theorem.
  \end{proof}

\section{Stability of the attractor to non-monotone perturbations}

\noindent In this section we study the sweeping process
\begin{equation}\label{1eps}
   -\dot x(t)\in N_{C(t)}(x(t))+f(t,x(t),\eps),
\end{equation}
which satisfies the monotonicity condition (\ref{monot}) only when $\eps=\eps_0,$ i.e.
\begin{equation}\label{monot0}
      \langle f(t,x_1,\eps_0)-f(t,x_2,\eps_0),x_1-x_2 \rangle  \geq \alpha \lVert x_1-x_2 \lVert ^2,\ \ \mbox{for some fixed}\ \alpha>0\ \mbox{and for all}\ t\in\mathbb{R},\ x_1,x_2\in\mathbb{R}^n.
\end{equation}

\subsection{The case where the dependence of the perturbation on the parameter $\eps$ is continuous}

\noindent Here we assume that
\begin{equation}\label{Lip4}
  \|f(t_1,x_1,\eps)-f(t_2,x_2,\eps)\|\le L_f\|t_1-t_2\|+L_f\|x_1-x_2\|, \quad {\rm for\ all\ }t_1,t_2\in\mathbb{R},\ x_1,x_2\in\mathbb{R}^n.
\end{equation}

\begin{theorem}\label{thm4} Let $f:\mathbb{R}\times\mathbb{R}^n\times\mathbb{R}\to\mathbb{R}^n$ satisfy the Lipschitz condition  (\ref{Lip4}) and the monotonicity condition (\ref{monot0}).
Assume that, for any $t\in\mathbb{R}$, the set $C(t)\subset\mathbb{R}^n$ is nonempty, closed, convex and the uniformly bounded map $t\mapsto C(t)$ satisfies the Lipschitz condition (\ref{Lip}).  Finally, assume that  $f(t,x,\eps)$ is continuous at $\eps=\eps_0$ uniformly in $t\in\mathbb{R},$ $x\in\mathbb{R}^n.$ 
Let $x_0:\mathbb{R}\to\mathbb{R}^n$ be the unique solution to (\ref{1eps}) with $\eps=\eps_0$ provided by Theorem~\ref{thm11}.
Then, given any $\gamma>0$ there exists $t_1\in\mathbb{R}$ such that for any solution $x_\eps$ of (\ref{1eps}) defined on $[0,\infty),$ one has
\begin{equation}\label{i}
   \|x_\eps(t)-x_0(t)\|<\gamma,\qquad t\ge t_1,
\end{equation}
for all $\eps$ sufficiently close to $\eps_0.$
\end{theorem}

\noindent We remind the reader that coresponding results for differential inclusions with bounded right-hand-sides are known e.g. from Kloeden-Kozyakin \cite{kloeden1}.

\vskip0.2cm

\noindent The following lemma will be used iteratively throughout the rest of the paper.

\begin{lem}\label{mainlem} Let $x_\eps$ be a solution of (\ref{1eps}) defined on $[\tau,\infty)$. Let $x_0=x_{\eps_0}.$ If (\ref{monot0}) holds, then, for a.a. $t\ge \tau$, 
	\begin{equation}\label{estimate1}
	\lVert x_{\eps}(t) - x_0(t) \rVert^2 \leq e^{-2\alpha(t-\tau)}\lVert x_{\eps}(\tau) - x_0(\tau) \rVert^2 - 2\int_{\tau}^{t}e^{-2\alpha(t-s)}\langle f(s,x_\eps(s),\eps) - f(s,x_\eps(s),\eps_0) ,x_\eps(s) -x_0(s)   \rangle ds.
	\end{equation}
\end{lem}

\begin{proof}
For a.a. $t\ge\tau$ and $\eps\in\mathbb{R}$ we have 
	\begin{align*}
	\frac{d}{dt} \lVert x_\eps(t) -x_0(t) \rVert^2 &=2 \left< \dot{x_\eps}(t) -\dot{x_0}(t) , x_\eps(t) -x_0(t) \right>\\
	&\leq 2\left< -f\left(t , x_\eps(t),\eps\right) ,  x_\eps(t) -x_0(t) \right> + 2\left< f(t,x_0(t),\eps_0) , x_\eps(t) -x_0(t) \right>\\
	&= -2 \left\langle f(t,x_\eps(t),\eps) - f(t,x_\eps(t),\eps_0) ,  x_\eps(t) -x_0(t) \right\rangle -2\langle f(t,x_\eps(t),\eps_0) -f(t,x_0(t),\eps_0) , x_\eps(t) -x_0(t) \rangle \\ 
	&\leq -2\alpha \lVert x_{\eps}(t) - x_0(t) \rVert^2 - 2\langle f(t,x_\eps(t),\eps) - f(t,x_\eps(t),\eps_0) ,x_\eps(t) -x_0(t)   \rangle
	\end{align*}
and the conclusion follows by applying the Gronwall-Bellman lemma (see Lemma~\ref{trub} in the Appendix).
\end{proof}

\noindent {\it Proof of Theorem~\ref{thm4}.} By Lemma~\ref{mainlem} and (\ref{M}) one has
\begin{equation}\label{ii}
	\lVert x_{\eps}(t) - x_0(t) \rVert^2 \leq e^{-2\alpha t}\lVert x_{\eps}(0) - x_0(0) \rVert^2 +\left(\dfrac{1}{2\alpha}-\dfrac{e^{-2\alpha t}}{2\alpha}\right)\max\limits_{s\in[0,t]}\|f(s,x_\eps(s),\eps)-f(s,x_\eps(s),\eps_0)\|\cdot M,
\end{equation}
from which the conclusion follows.\qed

\begin{rem} The estimate (\ref{i}) can be extended to the entire $\mathbb{R},$ if $x_\eps$ is defined on the entire $\mathbb{R}$ (e.g. if $x_\eps$ is that given by Theorem~\ref{thm1}). Indeed, in this case (\ref{ii}) can be strengthened to
$$	\lVert x_{\eps}(t) - x_0(t) \rVert^2 \leq e^{-2\alpha (t-\tau)}\lVert x_{\eps}(\tau) - x_0(\tau) \rVert^2 +\left(\dfrac{1}{2\alpha}-\dfrac{e^{-2\alpha (t-\tau)}}{2\alpha}\right)\max\limits_{s\in[\tau,t]}\|f(s,x_\eps(s),\eps)-f(s,x_\eps(s),\eps_0)\|\cdot M,
$$
which gives
$$	\lVert x_{\eps}(t) - x_0(t) \rVert^2 \leq \dfrac{1}{2\alpha}\max\limits_{s\in(-\infty,t]}\|f(s,x_\eps(s),\eps)-f(s,x_\eps(s),\eps_0)\|\cdot M,
$$
by passing to the limit as $\tau\to-\infty.$
\end{rem}

\subsection{The case where the dependence of the perturbation on the parameter $\eps$ is only integrally continuous}

\noindent In this section we assume that the following version of Lipschitz condition (\ref{Lipf}) holds:
\begin{equation}\label{Lipfeps}
\begin{array}{ll}
  \|f(t_1,x,\eps)-f(t_2,x,\eps)\|\le L_\eps\|t_1-t_2\|,& {\rm for\ all\ }t_1,t_2\in\mathbb{R},\ x\in\mathbb{R}^n,\ \eps\in\mathbb{R}\backslash\{\eps_0\},\\
 \|f(t,x_1,\eps)-f(t,x_2,\eps)\|\le L_f\|x_1-x_2\|,& {\rm for\ all\ }t\in\mathbb{R},\ x_1,x_2\in\mathbb{R}^n,\ \eps\in\mathbb{R},
\end{array}
\end{equation}
where $L_\eps>0$ may depend on $\eps\in\mathbb{R}$ and $L_f>0$ is independent of $\eps\in\mathbb{R}^n.$ 
Following Krasnoselskii-Krein \cite{kk} and Demidovich \cite[Ch.~V,\ \S3]{dem}, we say that $f(t,x,\eps)$ is {\it integrally continuous} at $\eps=\eps_0$, if
\begin{equation}\label{ic}
   \lim\limits_{\eps\to\eps_0}\int\limits_{\tau}^{t} f(s,x,\eps)ds=\int\limits_{\tau}^{t} f(s,x,\eps_0)ds,\qquad \mbox{for all }\tau,t\in \mathbb{R},\ x\in \mathbb{R}^n.
\end{equation}
The central role in this section is played by a  generalization of the theorem on passage to the limit in the integral due to Krasnoselskii-Krein \cite{kk} (see also Demidovich \cite[Ch.~V,\ \S3]{dem}). We will formulate this theorem for the case when $f(t,x,\eps)$ satisfies the Lipschitz condition (\ref{Lipfeps}).
\begin{theorem}\label{kkthm} {\bf (Krasnoselskii-Krein \cite{kk})} Assume that $F:\mathbb{R}\times\mathbb{R}^k\times\mathbb{R}\to\mathbb{R}^n$ satisfies (\ref{Lipfeps}) and that $t\mapsto F(t,u,\eps_0)$ is continuous for every $u\in\mathbb{R}^k.$ Consider a family of continuous functions $\{u_\eps(t)\}_{\eps\in\mathbb{R}}$ defined on an interval $[\tau,T]$ such that $u_\eps(t)\to u_0(t)$ uniformly on $[\tau,T].$ If $F$ verifies the integral continuity property (\ref{ic}), then
$$
  \lim\limits_{\eps\to\eps_0}\int\limits_\tau^t F(s,u_\eps(s),\eps)ds=\int\limits_\tau^t F(s,u_0(s),\eps_0)ds,\qquad \mbox{for\ all\ }t\in[\tau,T].
$$
\end{theorem}
\noindent In this statement, we take $k=n$ when referring to (\ref{Lipfeps}) and (\ref{ic}) in the context of the function $F$. 

\vskip0.2cm

\noindent We are now in the position to prove the main result of this section.

\begin{theorem}\label{icthm} Let $f:\mathbb{R}\times\mathbb{R}^n\times\mathbb{R}\to\mathbb{R}^n$ satisfy the Lipschitz condition  (\ref{Lipfeps}). Assume that $f$ satisfies the monotonicity condition (\ref{monot0}).
Assume that, for any $t\in\mathbb{R}$, the set $C(t)\subset\mathbb{R}^n$ is nonempty, closed, convex and the uniformly bounded map $t\mapsto C(t)$ satisfies the Lipschitz condition (\ref{Lip}). Finally, assume that   $f(t,x,\eps)$ is integrally continuous at $\eps=\eps_0$. Then, given any $\gamma>0$ there exists $t_1\ge 0$ such that for any solution $x_\eps $ to (\ref{1eps}) defined on $[0,\infty)$ and for any $t_2\ge t_1$, one has
$$
   \|x_\eps(t)-x_0(t)\|<\gamma,\qquad t\in[t_1,t_2],
$$
for all $\eps$ sufficiently close to $\eps_0.$
\end{theorem}

%\noindent Note, the existence of at least one bounded solution $x_\eps$ to (\ref{1eps}) for any $\eps\in\mathbb{R}$ is provided by Theorem~\ref{thm1}.

\begin{proof} 
Let us fix some closed interval $[t_1,t_2]$ and assume that the statement of the theorem is wrong, i.e. assume that there exists $\gamma>0$ such that \begin{equation}\label{contr}
\max\limits_{t\in[t_1,t_2]}\|x_{\eps_m}(t)-x_0(t)\|\ge \gamma
\end{equation} for some sequence $\eps_m\to\eps_0$ as $m\to\infty.$ By (\ref{M}), we can find $\tau<0$ such that 
\begin{equation}\label{conc1}
   e^{-2\alpha(t-\tau)}\lVert x_{\eps_m}(\tau) - x_0(\tau) \rVert^2<\dfrac{\gamma}{2},\quad{\rm for\ all\ }m\in\mathbb{N},\ t\in[t_1,t_2].
\end{equation}
In what follows, we show that the integral term of the estimate (\ref{estimate1}) can be made smaller that $\gamma/2$ on the sequence $x_{\eps_m}$ as well.
Since $f(t,x,\eps)$ is uniformly bounded and  $C$ satisfies the global Lipschitz condition (\ref{Lip}), by Edmond-Thibault \cite[Theorem~1]{existence} we have the existence of $L_0>0$ such that 
$$
   \|\dot x_{\eps_m}(t)\|\le L_0,\qquad\mbox{for all }m\in\mathbb{N},{\rm \ and \ a.a.\ }t\in[\tau,T]
$$where $T>0$. Since the functions of $\{x_{\eps_m}(t)\}_{m\in\mathbb{N}}$ are uniformly bounded according to (\ref{M}), the Ascoli-Arzela theorem implies that without loss of generality the sequence $\{x_{\eps_m}(t)\}_{m\in\mathbb{N}}$ can be assumed convergent uniformly on $[\tau,T].$ Introduce
$$
  F(t,(x_1,x_2)^T,\eps)=\left<f(t,x_1,\eps)-f(t,x_1,\eps_0),x_2\right>,\quad u_m(t)=\left(x_{\eps_m}(t),e^{2\alpha t}\left(x_{\eps_m}(t)-x_0(t)\right)\right)^T,\quad    %x_*(t)=\lim\limits_{m\to\infty} x_{\eps_m}(t),
$$
so that $F:\mathbb{R}\times\mathbb{R}^{2n}\times\mathbb{R}\to\mathbb{R}^n$.
Since $f(t,x,\eps)$ is integrally continuous at $\eps=\eps_0$, then
$$
   \lim\limits_{\eps\to\eps_0}\int\limits_\tau^t F\left(s,(x_1,x_2)^T,\eps\right)ds = 0,\quad{\rm for\ all\ } (x_1,x_2)^T\in\mathbb{R}^{2n},\ t\in[\tau,T].
$$
Furthermore, the function $F$ satisfies the same type of Lipschitz condition (\ref{Lipfeps}) as $f$ does.
The Krasnoselskii-Krein theorem (Theorem~\ref{kkthm}), therefore, implies 
\begin{equation}\label{conc2}
  \lim\limits_{m\to\infty}\int\limits_\tau^t F(s,u_m(s),\eps_m)ds=0,\quad{\rm for\ all\ }t\in [\tau,T].
\end{equation}
The conclusions (\ref{conc1}) and (\ref{conc2}) contradict (\ref{contr}) because of (\ref{estimate1}). The proof follows by Lemma~\ref{mainlem}.
\end{proof}

\subsection{A particular case: high-frequency vibrations} 

\noindent In this section we consider a sweeping process 
\begin{equation}\label{hf}
   -\dot x(t)\in N_{C(t)}(x(t))+g\left(\dfrac{t}{\eps},x(t)\right),
\end{equation}
where both $t\mapsto C(t)$ and $t\mapsto g(t,x)$ are almost periodic and we use 
Theorem~\ref{icthm} in order to estimate the location of solutions of (\ref{hf}) 
for large values of time and for small values of $\eps.$ 

\vskip0.2cm

\noindent Since $g$ is almost periodic in the first variable, the following property holds uniformly in $a\in\mathbb{R}$ (see Bohr \cite[p.~44]{bohr})  
\begin{equation}\label{g0}
  g_0(x)=   \lim\limits_{T\to\infty}\dfrac{1}{T}\int\limits_0^{T}g(\tau,x)d\tau=   \lim\limits_{T\to\infty}\dfrac{1}{T}\int\limits_a^{T+a}g(\tau,x)d\tau,
\end{equation}
where both limits exist. Therefore,
$$
   \lim\limits_{\eps\to 0}\int\limits_{\tau}^t g\left(\dfrac{s}{\eps},x\right)ds=\lim\limits_{T\to \infty}(t-\tau)\dfrac{1}{T}\int\limits_{\tau T/(t-\tau)}^{T+\tau T/(t-\tau)}g(s,x)ds=\int\limits_\tau^t g_0(x) ds.
$$
By the other words,
the function
$$
  f(t,x,\eps)=\left\{\begin{array}{ll}
     g\left(\dfrac{t}{\eps},x\right), & {\rm if}\ \eps\not=0,\\
     g_0(x), & {\rm if}\ \eps=0,\end{array}\right.
$$
is integrally continuous at $\eps=0$ in the sense of (\ref{ic}). We arrive to following corollary of Theorems \ref{thm2} and \ref{icthm}.

\begin{cor}\label{coro} Let $t\mapsto C(t)$ be an almost periodic function that satisfies the global Lipschitz condition (\ref{Lip}). Assume that, for each $x\in\mathbb{R}^n$, the function $t\mapsto g(t,x)$ is almost periodic and satisfies the global Lipschitz condition
$$
   \left\|g(t_1,x_1)-g(t_2,x_2)\right\|\le L_g|t_1-t_2|+L_g\|x_1-x_2\|,\qquad {\rm for\ all\ }t_1,t_2\in\mathbb{R},\ x_1,x_2\in\mathbb{R}^n.
$$
Finally, assume that for some $\alpha>0$ the function $g_0$ given by (\ref{g0}) satisfies the monotonicity condition
$$
   \left<g_0(x_1)-g_0(x_2),x_1-x_2\right>\ge \alpha\|x_1-x_2\|^2,\qquad {\rm for\ all\ }x_1,x_2\in\mathbb{R}^n.
$$ If $x_\eps$ is any solution of (\ref{hf}) defined on $[0,\infty)$, then
uniformly on any time-interval $[t_1,t_2]$ with sufficiently large $t_1$, the family $\{x_\eps(t)\}_{\eps\in\mathbb{R}}$ converges, as $\eps\to 0$, to the unique globally exponentially stable almost periodic solution $x_0(t)$ of the averaged sweeping process
$$
   -\dot x(t)\in N_{C(t)}(x(t))+g_0(x(t)).
$$
\end{cor}

\subsection{Instructive examples}\label{examples}
\noindent The examples of this section illustrate how the results of the paper are supposed to be used in applications.
\begin{ex}\label{exam1}
\noindent Consider a one-dimensional sweeping process
\begin{equation}\label{ex1}
   -\dot x(t)\in N_{[\sin(t),\sin(t)+1]}(x(t))+\eps x^2(t)+\left(\sin\left(\sqrt{2}\cdot t\right)+2 \right) x(t).
\end{equation}
 The sweeping process (\ref{ex1}) satisfies the monotonicity property (\ref{monot}) when $\eps=0.$ Theorems \ref{thm2} and \ref{thm4}  imply that for any $\gamma>0$ there exists $t_1>0$ such that any solution $x_\eps$ of (\ref{ex1}) with $x_\eps(0)\in [0,1]$ satisfies $\|x_\eps(t)-x_0(t)\|\le\gamma$ for all $t\ge t_1$ and for all $|\eps|$ sufficiently small, where $x_0$ is the unique globally exponentially stable almost periodic solution to
$$   -\dot x(t)\in N_{[\sin(t),\sin(t)+1]}(x(t))+\left(\sin\left(\sqrt{2}\cdot t\right)+2 \right)x(t).
$$ \end{ex}

\begin{ex}\label{exam2} Let us now show that the monotonicity of a sweeping process gets broken by a high-frequency ingredient as follows
\begin{equation}\label{ex2}
   -\dot x(t)\in N_{[\sin(t),\sin(t)+1]}(x(t))+\sin\left(\dfrac{t}{\eps}\right) x^2(t)+\left(\sin\left(\sqrt{2}\cdot t\right)+2 \right)x(t).
\end{equation}
The non-monotonic term $\sin\left(\dfrac{t}{\eps}\right)$ no longer approaches 0 as it took place in Example~\ref{exam1} and Theorem~\ref{thm4} is inapplicable. However, $\sin\left(\dfrac{t}{\eps}\right)$ approaches 0 as $\eps\to 0$ integrally (i.e. in the sense of (\ref{ic})) on any bounded time interval $[t_1,t_2]$. Therefore, Corollary~\ref{coro} ensures that given any $\gamma>0$ there exists $t_1>0$ such that for any $t_2>t_1$ and for any solution $x_\eps$ of (\ref{ex2}) with $x_\eps(0)\in[0,1]$ one has $\|x_\eps(t)-x_0(t)\|\le\gamma$ on $[t_1,t_2]$ for all $|\eps|$ sufficiently small, where $x_0$ is the unique globally exponentially stable almost periodic solution to the averaged sweeping process $$
 -\dot x(t)\in N_{[\sin(t),\sin(t)+1]}(x(t))+\left(\sin\left(\sqrt{2}\cdot t\right)+2 \right)x(t).
$$
\end{ex}
 
\noindent To summarize, Examples~\ref{exam1} and \ref{exam2} establish useful qualitative properties of  non-monotone sweeping processes without any need of actual computing of solutions. Numerical computation of solutions of (\ref{ex1}) and (\ref{ex2}) (e.g. using the catch-up algorithm of Edmond-Thibault \cite{existence}) is thus outside the scope of this paper.

\subsection{Applications in elastoplasticity} \label{newsec}

\noindent The perturbation term of the sweeping processes that model networks of elastoplastic springs (like those in Bastein et al \cite{lam2}) does not generally satisfy the monotonicity property (\ref{monot}) because it always contains oscillatory terms coming from springs. One can expect monotonicity (caused by viscous friction) only when the eigenfrequencies of all springs vanish, which suggests that magnitudes of these eigenfrequencies is a natural choice for the small parameter $\eps$ as long as applications of Theorems \ref{thm4} and \ref{icthm} in elastoplasticity are concerned. The eigenfrequencies of the springs can be viewed small compared to other parameters, if the masses of nodes of the network (i.e. inertial forces) are large. 
However, setting the so-selected small parameter $\eps$ to 0 will ensure monotonicity and global asymptotic stability for velocity-like variables only, not for the position-like variables. This can be intuitively seen from a simple oscillator $\ddot x+c\dot x+\eps h(t,x)=0$, whose solutions approach 
those of  $\ddot x+c\dot x=0$ as $\eps\to 0$ (assuming that $h(t,x)$ stays bounded).
The solutions of the reduced oscillator asymptotically approach the line $\mathbb{R}\times\{0\}$ because of the monotonicity provided by the friction term. As a consequence, the solutions of the original oscillator stay close to $\mathbb{R}\times\{0\}$ for small values of $\eps>0.$ Coming back to the sweeping processes of elastoplasticity, we expect that for large inertial forces the methods of Theorems \ref{thm4} and \ref{icthm}  will predict convergence to the manifold of  equilibria that correspond to infinite inertial forces. Pursuing this plan is a subject of a different paper, that we are working on.

\section{Conclusion} 

\noindent In this paper we established the existence and global exponential stability of bounded and almost periodic solutions of the Moreau sweeping process (\ref{1}). We proved that non-monotone sweeping processes with bounded right-hand-sides admit at least one solution defined on the entire $\mathbb{R}.$ When the sweeping process satisfies the monotonicity property (\ref{monot}), we proved the existence of exactly one bounded solution defined on $\mathbb{R}$ which is almost periodic when the right-hand-sides of (\ref{1}) are almost periodic. 
\vskip0.2cm

\noindent When the right-hand-sides of (\ref{1}) are non-monotone, but close to monotone, we discovered that all the solutions to (\ref{1}) are close to the unique bounded (or almost periodic) solution of the respective monotone process for large values of time. In particular, we initiated the development of the averaging theory for Moreau sweeping process (\ref{1}) with high-frequency almost periodic  excitation $g\left(\dfrac{t}{\eps},x\right)$, where only monotonicity of the average 
 $g_0(x)=\lim\limits_{T\to\infty}\dfrac{1}{T}\int\limits_0^T g(s,x)ds$ is required.
 This result can be used for the design of  vibrational control strategies for Moreau sweeping processes (see e.g. Bullo \cite{bullo} for the respective theory in the case of differential equation).

\vskip0.2cm

\noindent The approach of this paper finds applications in the problem of global stabilization of parallel networks of elastoplastic springs where the period of the mechanical loading (e.g. stretching/compressing) of springs  doesn't coincide with the period of the force that excites the masses at nodes, as we discussed in the Introduction and in Section~\ref{newsec}.

\vskip0.2cm

\noindent Further potential applications of the results of this paper are in studying the dynamics of  a circuit involving devices like diodes, thyristors and  diacs  (see Addi  et al \cite{hybrid}) when ampere-volt characteristics (for the set function) and  voltage supply (for the perturbation) receive time-periodic excitations of different periods. Such a study will require extending our theory to sweeping processes with state-dependent convex constraints.

\section{Appendix}

\noindent The following version of Gronwall-Bellman lemma and its proof are taken from Trubnikov-Perov \cite[Lemma~1.1.1.5]{trub}. 

\begin{lem}\label{trub} {\bf (Gronwall-Bellman)} Let an absolutely continuous function $a:[0,T]\to\mathbb{R}$ satisfy
\begin{equation}\label{1.1.1.21}
   \dot a \le \lambda a(t)+b(t),\qquad{\rm for\ a.a.\ }t\in[0,T],
\end{equation}
where $b:[0,T]\to\mathbb{R}$ is an integrable function. Then
$$
   a(t)\le e^{\lambda t}a(0)+\int\limits_0^t e^{\lambda (t-s)}b(s)ds,\qquad{\rm for\ all\ }t\in[0,T].
$$
\end{lem}
\begin{proof} By introducing
$$
   \psi(t)=e^{\lambda t} a(0)+\int\limits_0^t e^{\lambda (t-s)}b(s)ds,
$$
one has
$$
  \psi(t)e^{-\lambda t}-\int\limits_0^t e^{-\lambda s}b(s)ds=a(0)
$$
and so
$$
  \dfrac{d}{dt}\left[\psi(t)e^{-\lambda t}-\int\limits_0^t e^{-\lambda s}b(s)ds\right]=0,\qquad{\rm for\ a.a.\ }t\in[0,T],
$$
which implies
$$
  \dot \psi(t)-\lambda \psi(t)=b(t)\ge \dot a(t)-\lambda a(t).
$$
If now
$$
  u(t)=a(t)-\psi(t),
$$
then $\dot u(t)\le \lambda u(t)$ and so $\dfrac{d}{dt}\left[u(t)e^{-\lambda t}\right]=e^{-\lambda t}(\dot u-\lambda u)\le 0,$ i.e. $u(t)e^{-\lambda t}\le u(0)$. Therefore, $u(t)\le 0$ and
$$
  a(t)\le\psi(t)=e^{\lambda t}a(0)+\int\limits_0^t e^{\lambda (t-s)} b(s)ds.
$$
\end{proof}

\noindent The following proof is known (see e.g. \cite[Theorem~8.7]{leine} and \cite[Lemma~2]{leine2}), but we add a proof in terms of sweeping process (\ref{1}) for completeness.

\vskip0.2cm

\noindent {\it Proof of Theorem~\ref{thm11}}.
\noindent {\bf Step 1:} {\it Incremental stability}. Let $x_1$ and $x_2$ be solutions to (\ref{1}) with the initial conditions $x_1(t_0),x_2(t_0)\in C(t_0)$. Assuming that $t\ge t_0$ is such that both $\dot x_1(t)$ and $\dot x_2(t)$ exist and verify  (\ref{1}), one has
  	 \begin{equation*}
  	 \langle -\dot{x}_1(t) - f(t,x_1(t)), x_1(t) - x_2(t) \rangle \geq 0.
  	\end{equation*}
  	 Therefore $ \langle - f(t,x_1(t)), x_1(t) - x_2(t) \rangle \geq  \langle \dot{x}_1(t) , x_1(t) - x_2(t) \rangle $. By analogy, 
  	  $-\dot{x}_2(t) - f(t,x_2(t)) \in N_{C(t)}(x_2(t))$ implies  	 
  	 $
  	 \langle -\dot{x}_2(t) , x_1(t) - x_2(t) \rangle \leq \langle  f(t,x_2(t)), x_1(t) - x_2(t) \rangle
  	 $   
  	 Therefore,
  	 \begin{align*}
  	 \frac{d}{dt} \lVert x_1(t)-x_2(t) \lVert ^2 
  	 &= 2 \langle \dot{x}_1(t) - \dot{x}_2(t) , x_1(t)-x_2(t)\rangle \\
  	 &= 2 \langle \dot{x}_1(t), x_1(t)-x_2(t) \rangle -2 \langle
  	 \dot{x}_2(t) , x_1(t)-x_2(t) \rangle \\
  	& \leq -2 \langle  f(t,x_1(t)) , x_1(t)-x_2(t)\rangle + 2 \langle f(t,x_2(t)), x_1(t)-x_2(t) \rangle \\ 
  	&= -2\langle f(t,x_1(t))-f(t,x_2(t)), x_1(t)-x_2(t)\rangle\\
  	& \leq -2 \alpha \lVert x_1(t)-x_2(t) \lVert^2
  	 \end{align*}
  	and by Gronwall-Bellman lemma (see Lemma~\ref{trub} in the Appendix), 
$  	\lVert x_1(t)-x_2(t) \lVert ^2 \leq e^{-2\alpha(t-t_0)} \lVert x_1(t_0)-x_2(t_0) \lVert^2,$ for a.a. $t\ge t_0.$
Since both $x_1$ and $x_2$ are continuous functions,
  	\begin{equation}\label{Gro}
  	\lVert x_1(t)-x_2(t) \lVert ^2 \leq e^{-2\alpha(t-t_0)} \lVert x_1(t_0)-x_2(t_0) \lVert^2,\quad\mbox{for all }t\ge t_0.
  	\end{equation} 

\vskip0.2cm

\noindent {\bf Step 2.} {\it Uniqueness of the bounded solution $x_0$}.  Let $v$ be another bounded solution of (\ref{1}) defined on the entire $\mathbb{R}.$ Then, given any $\tau\in\mathbb{R},$ the inequality (\ref{Gro}) yields
\begin{equation*}
  	\lVert x_0(t)-v(t) \lVert ^2 \leq e^{-2\alpha(t-\tau)} \lVert x_0(\tau) - v(\tau) \lVert^2,\quad \text{ for all } t\geq \tau.
  	\end{equation*} 
  	Thus $\lVert x_0(t)-v(t) \lVert  \leq 2 M e^{-\alpha(t-\tau)}$, for all $t\geq \tau$, where $M$ is as defined in (\ref{M}). Now we fix $t\in\mathbb{R}$ and pass to the limit as $\tau\to-\infty,$ obtaining   $ \lVert u(t) - v(t) \lVert^2\le 0.$ Thus $u(t)=v(t)$ for all  $t\in \mathbb{R}$.

\vskip0.2cm

\noindent {\bf Step 3.} {\it Global exponential stability of $x_0$} follows from  (\ref{Gro}). Indeed, (\ref{Gro}) implies that 
$\lVert x_0(t)-v(t) \lVert  \leq e^{-\alpha(t-\tau)} \lVert x_0(\tau) - v(\tau) \lVert $, for any solution $v$ of (\ref{1}) and for any $t\ge\tau$. 	
\qed

\section{Acknowledgements.}  The work of the first author is supported by the Ministry of Education and Science of the Russian Federation in the frameworks of the project part of the state work quota (Project No 1.3464.2017) and by RFBR grant 16-01-00386. The second author acknowledges support by NSF Grant CMMI-1436856.

%% References
%% \cite{key} ==>> [#]
%% \cite[chap. 2]{key} ==>> [#, chap. 2]
%\cite{ISI:000288508200001}
%\bibliographystyle{elsarticle-num}
%\nocite{*}
\bibliographystyle{plain}
%\bibliography{literature}

\end{document}